\documentclass[12pt]{article}
\usepackage{geometry} % see geometry.pdf on how to lay out the page. There's lots.

\usepackage[affil-it]{authblk}
\usepackage{amsmath}
 \usepackage{amssymb}
\usepackage{amsthm}
\usepackage{subfigure}
\usepackage{tikz}
\usepackage{color}
\usepackage{cite}

 \newtheorem{theorem}{Theorem}[section]
\newtheorem{lemma}[theorem]{Lemma}
\newtheorem{proposition}[theorem]{Proposition}

\theoremstyle{definition}
\newtheorem{example}[theorem]{Example}
\newtheorem{examples}[theorem]{Examples}
\newtheorem{definition}[theorem]{Definition}

\numberwithin{equation}{section}

\newcommand{\larc}[1]{\hspace{-.4ex}\overset{#1}{\frown}\hspace{-.4ex}}

\author{Carolina Benedetti\thanks{NSERC and CRC (N. Bergeron's ). Partially supported by Beca Mazda para el Arte y la Ciencia.}}
\title{Combinatorial Hopf algebra of superclass functions of type $D$}
\date{}

\begin{document}
\maketitle
\begin{abstract}
We provide a Hopf algebra structure on the space of superclass functions on the unipotent upper triangular group of type $D$ over a finite field based on a supercharacter theory constructed by Andr\'e and Neto in \cite{andre_2,andre}. Also, we make further comments with respect to types $B$ and $C$. Type $A$ was explored in  \cite{aim} by M. Aguiar et. al; thus this paper is a contribution to understand combinatorially the supercharacter theory of the other classical Lie types. 
 \end{abstract}

\section{Introduction}
\label{sec:in}

The problem of simultaneously reducing to canonical form two linear operators on a finite-dimensional space is a ``wild" problem in representation theory. This problem contains all classification matrix problems given by quivers (see \cite{wild}). In this sense, the classical representation theory for the type $A$ group $U_n(q)$ of unipotent $n\times n$ upper triangular matrices over a finite field is known to be wild. This makes, in some sense, hopeless any attempt to study the representation theory of the group $U_n(q)$. In his Ph.D. thesis C. Andr\'e started to develop a theory that approximates the representation theory of $U_n(q)$. Roughly speaking, by using certain linear combinations of irreducible characters and lumping together conjugacy classes under certain conditions, the resulting theory behaves very nicely (see \cite{andre_96, thiem}). This gave rise to the concept of  ``supercharacter theory". Later on, P. Diaconis and I. M. Isaacs extended this concept to arbitrary algebra groups (see \cite{diaconis}). Supercharacter theory of the group $U_n(q)$ has a rich combinatorics which connects this beautiful theory with classical combinatorial objects. As a matter of fact, in \cite{aim} a Hopf algebra structure is provided on the graded vector space \textbf{SC} of superclass characteristic functions over $U_n(q)$, for $n\geq 0$. Moreover, when $q=2$ this Hopf algebra is a realization of a well-known combinatorial Hopf algebra, namely, the Hopf algebra of symmetric functions in noncommuting variables (see \cite{sagan, bergeron_98}).

The reader familiar with the classical representation theory of the symmetric group $S_n$ will notice how this resembles the relationship between symmetric functions and the character theory of $S_n$. Also, supercharacters of $U_n(2)$ are indexed by set partitions of the set $[n] = \{1,2,\dots ,n\}$ and by \emph{labelled} set partitions for general $q$.

In this paper, we study combinatorially the supercharacter theory corresponding to the other classical Lie types $B,\; C$ and $D$, making emphasis on the latter. This study is based on the supercharacter theory constructed by Andr\'e and Neto in \cite{andre_06,andre_2}. These groups fail to be algebra groups unlike type $A$. However, we can regard them as subgroups of the convenient group of type $A$ and restrict the supercharacter theory of type $A$ to the respective subgroup. 

The paper is organized as follows. In Section \ref{sec_pre} we provide the reader with the basic definitions concerning  combinatorial Hopf algebras and supercharacters. In Section \ref{sec_3}, we give a combinatorial interpretation for the supercharacter theory of the group $U^{D}_{2n}(q)$ of even orthogonal unipotent upper triangular matrices with coefficients in the field $\mathbb F_q$ of characteristic $\geq 3$. This combinatorial interpretation uses \emph{labelled} $D_{2n}$\emph{-partitions} of the set $[\pm n] := \{1,\dots , n,-n,\dots ,-1\} $. More specifically, we use these partitions to index orbit representatives for superclasses and supercharacters of the group $U_{2n}^D(q)$. Then, we define the analog of $\textbf{SC}$ for type $D$ as follows:
\begin{align*}
\textbf{SC}^D &= \oplus_{n\geq 0}\textbf{SC}_{2n}^D\\
&= \oplus_{n\geq 0}\text{span}_{\mathbb C}\{\kappa _{\lambda}:\lambda\in D_{2n}(q)\}
\end{align*}
where $\kappa _{\lambda}$ denotes the superclass characteristic function indexed by the labelled $D_{2n}$-partition $\lambda$. Using a change of basis, we prove that the space $\textbf{SC}^D$ is endowed with a Hopf algebra structure. This Hopf algebra structure is in analogy to the one given for type $A$ in \cite{aim}. However, the product structure on $\mathbf{SC}^D$ is not raised directly from representation theory. The coalgebra structure is raised directly from representation theory by using restriction.

In the final section, we discuss briefly the supercharacter theory for types $B$ and $C$. Also we make some remarks concerning forthcoming work on the Hopf monoid structure that $\textbf{SC}^D$ carries, following the results in \cite{aguiar:2011}.

\section{Preliminaries}\label{sec_pre} 

We start by defining supercharacter theory for a finite group $G$. This definition, which can be stated in different ways, is due to Diaconis and Isaacs \cite{diaconis}.

\begin{definition}\label{superchar_def}
A \emph{supercharacter theory for $G$} consists of:

\begin{itemize}

\item A partition $\mathcal K$ of $G$
\item A set $\mathcal X$ of characters of $G$

\end{itemize}

such that the following holds:

\begin{itemize}
\item[1.]  $|\mathcal K|=|\mathcal X |$
\item[2.] every irreducible character of $G$ is a constituent of a unique $\chi\in\mathcal X$
\item[3.] the characters in $\mathcal X$ are constant on members of $\mathcal K$.

\end{itemize}
\end{definition}
\noindent The elements in $\mathcal X$ are scalar multiples of linear combinations of the form $\sum_{\psi\in X}\psi(1)\psi$ where $X$ is a subset of irreducible characters of $G$, by \cite[Lemma 2.1]{diaconis}. 

\vspace{.3cm}
\noindent\emph{Remark:} Definition \ref{superchar_def} is equivalent to say that \begin{math}\text{span}_{\mathbb C}\{\sum_{g\in K}g: K\in \mathcal K\}\end{math} is a subalgebra of $Z(\mathbb C G)$ with unit 1. Given such a partition $\mathcal K$, there exists a unique $\mathcal X$, up to isomorphism, with the desired properties.

\begin{examples}
\begin{itemize}
\item Every group is endowed with the supercharacter theory where the set of superclasses $\mathcal K$ consists of the usual conjugacy classes and the set of supercharacters $\mathcal X$ is formed by the irreducible characters of $G$. 
\item Similarly, the trivial supercharacter theory of $G$ is such that 
$\mathcal K=\{\{1\},\; G-\{1\}\}$ and $\mathcal \chi=\{1,\rho _G-1\}$, where $\rho _G$ is the regular representation. 
\item A less trivial example is given by the cyclic group of order $2^n$, where $n\geq 2$. It is not hard to see that lumping together the elements of $G$ by their order, gives a set of superclasses $K$, whose corresponding supercharacters are formed by adding together all the $d$-primitive roots of unity for each $d|n$.
\end{itemize}
\end{examples}

This paper explores the particular supercharacter theory constructed by Andr\'e and Neto in \cite{andre_06} of the classical group $U_{2n}^D(q)$ of $2n\times 2n$ unipotent upper triangular matrices of type $D$. Here, we refer to this construction as \emph{the} supercharacter theory of type $D$, since it is the one we are interested in. We regard $U_{2n}^D(q)$ as a subgroup of the group $U_{2n}(q)$ of $2n\times 2n$ unipotent upper triangular matrices, which is an \emph{algebra group} as defined below (see \cite{isaacs}).
\vspace{.3cm}

\begin{definition}

Let $J$ be a finite dimensional associative nilpotent $\mathbb F$-algebra and let $G$ denote the set consisting of  formal objects of the form $1+a$ where $a\in J$. Then $G=1+F$ is a group, where the multiplication is given by $(1+a)(1+b)=1+a+b+ab$. The group $G$ is the \emph{algebra group} based on $J$.
\end{definition}

\noindent As an example, denote by $\mathfrak u_n$ the algebra of nilpotent upper triangular matrices associated to the group $U_n(q)$. Then we see that $U_n(q)=I+\mathfrak u_n$, and thus $U_n(q)$ is an algebra group.

The supercharacter theory for the group $U_n(q)$ has a very nice combinatorial interpretation. Its superclasses are indexed by labelled set partitions of type $A$ as well as its supercharacters  (see \cite{aim}). 
In analogy with type $A$, in the next section we describe the supercharacter theory for the group $U_{2n}^D(q)$ using \emph{labelled $D_{2n}$-partitions}, though as mentioned in the introduction, $U_{2n}^D(q)$ is not an algebra group. Before that, we give a quick intro to combinatorial Hopf algebras. For a further reading on this topic, see \cite{abs_hopf}.

\subsection{Combinatorial Hopf algebras}\label{hopf_algebra}

Let $\mathcal A$ be a vector space over a field $\mathbb K$. We say that $\mathcal A$ is an associative algebra with unit $1$  if $\mathcal A$ has a linear map $m:\mathcal A\otimes\mathcal A\rightarrow\mathcal A$ such that $m\circ(m\otimes \text{Id})=m\circ(\text{Id}\otimes m)$ where $\text{Id}$ is the identity map in $\mathcal A$. The unit can also be associated with a linear map $u :  \mathbb K\rightarrow\mathcal A$ such that $ t\mapsto t\cdot 1$.
\noindent The maps $m$ and $u$ must be compatible in the sense that 
$$
m\circ(\text{Id}\otimes u) = m\circ(u\otimes\text{Id}) = \text{Id}
$$
\noindent On the other hand, a \emph{coalgebra} is a vector space $D$ over $\mathbb K$ with a coproduct $\Delta: D\otimes D\rightarrow D$ and a counit $\epsilon:D\rightarrow \mathbb K$ which are $\mathbb K$-linear maps. The coproduct must be coassociative in the sense that $
(\Delta\otimes \text{Id})\circ\Delta=(\text{Id}\otimes\Delta)\circ\Delta$ and must be compatible with $\epsilon$:
$$
(\epsilon\otimes \text{Id})\circ\Delta=(\text{Id}\otimes\epsilon)\circ\Delta=\text{Id}
$$

\noindent If an algebra $(\mathcal A,m,u)$ has also a coalgebra structure given by $\Delta,\epsilon$, we say that $\mathcal A$ is a \emph{bialgebra} if $\Delta,\epsilon$ are algebra homomorphisms.

\null

\begin{definition}\label{hopf_algebra}
A \emph{Hopf algebra} $\mathcal A$ is a bialgebra together with a linear map $S:\mathcal A\rightarrow\mathcal A$ called antipode. The map $S$ satisfies
$$
\sum_kS(a_k)b_k=\epsilon(a)\cdot 1=\sum_ka_kS(b_k) \text{ \;\;\; where\;\;\;   }\Delta(a)=\sum_ka_k\otimes b_k
$$
\end{definition}

\noindent We say that a bialgebra $\mathcal A$ is \emph{graded} if there exists a direct sum decomposition
$$
\mathcal A=\bigoplus_{k\geq 0}A_k
$$
such that $A_i\otimes A_j\subseteq A_{i+j}$, $u(\mathbb K)\subseteq A_0$, $\Delta (A_i)\subseteq\oplus_{i=0}^nA_i\otimes A_{n-i}$ and $\epsilon (A_n)=0$ for $n\geq 1$. Finally, we say that $\mathcal A$ is connected if $A_0\cong\mathbb K$.

There are few different notations of a combinatorial Hopf algebra in the literature, but in this paper we say that $\mathcal A$ is a \emph{combinatorial Hopf algebra} if $\mathcal A$ is a graded and connected bialgebra with antipode and such that $\mathcal A$ has a distinguished basis with positive structure constants \cite{bll}, i.e., a distinguished basis that multiplies/comultiplies positively. %The advantage of having a graded and connected Hopf algebra is that the uniqueness of the antipode is guaranteed by recursion. %The condition on the singled out basis can be dropped and instead add the existence of a singled out character $\zeta$. See \cite{abs_hopf} for more details.

\section{Supercharacter theory of type $D$}\label{sec_3}

The supercharacter theory of type $D$ this paper considers is due to Andr\'e and Neto \cite{andre_06,andre_2}. Here, we give a combinatorial interpretation of their algebro-geometric construction. From now on, $\mathbb F_q$ will denote a field of characteristic $p\geq 3$ and order $q=p^r$ for some integer $r\geq 1$. Also, denote by $\mathbb F^*_q$ the multiplicative group of nonzero elements of the field $\mathbb F^q$.  The group $U_{2n}^D(q)$ corresponds to even orthogonal unipotent upper triangular matrices with coefficients in $\mathbb F_q$ and can be described as the following set (see \cite{carter}):

\begin{equation*}
U^D_{2n}(q)=\Biggr\{\begin{pmatrix}
P & PQ\\0 & JP^{-t}J
\end{pmatrix} : P\in U_n(q),\; Q\in M_n(q),\; JQ^{t}J=-Q \Biggr\},
\end{equation*}
\noindent where $M_n(q)$ is the set of $n\times n$ matrices over $\mathbb F_q$ and $J$ is the $n\times n$ matrix with ones in the antidiagonal and zeros elsewhere.

We will drop the subindex $2n$ from our notation when the size of matrices is clear from context. In order to describe the supercharacters of $U^D(q)$, we make use of an associated nilpotent algebra $\mathfrak u^D(q)$. The algebra $\mathfrak u^D(q)$ is given by

\begin{equation*}
\mathfrak u^D(q)=\Biggr\{\begin{pmatrix}
R & Q\\0 & -JR^{t}J
\end{pmatrix} : R\in U_n(q)-I_n,\; Q\in M_n(q),\; JQ^{t}J=-Q \Biggr\}.
\end{equation*}

\noindent with $M_n(q)$ and $J$ as before. We will make use of the total order 
\begin{equation*}
1\prec \cdots\prec n\prec -n\prec\cdots\prec -1
\end{equation*}
to index the columns and rows of matrices in $U^D(q)$ and in $\mathfrak u^D(q)$, from left to right and top to bottom.

A vector space basis for $\mathfrak u^D(q)$ over $\mathbb F_q$ is given by the matrices $\{y_{\alpha}\}_{\alpha}$ where $\alpha$ runs over the set of positive roots $\Phi ^+$ of type $D$, given by 
\begin{equation*}
\Phi^+=\{e_i\pm e_j:1\leq i<j\leq n\}
\end{equation*}

\noindent and $y_{\alpha}$ denotes the matrix 
\begin{equation*}
y_{\alpha}=\left\{
  \begin{array}{l l}
e_{i,j}-e_{-j,-i} & \quad\text{if $\alpha=e_i-e_j$}\\
e_{i,-j}-e_{j,-i} & \quad\text{if $\alpha=e_i+e_j$}\\
\end{array}\right.
\end{equation*}
where $e_{i,j}\in\mathfrak u^D(q)$ has 1 in position $i,j$ and zeros elsewhere. Now define the support of $y_{\alpha}$ by
\begin{equation*}
\text{supp}(y_{\alpha})=\left\{
  \begin{array}{l l}
(i,j),\;(-j,-i) & \quad\text{if $\alpha=e_i-e_j$}\\
(i,-j),\;(j,-i) & \quad\text{if $\alpha=e_i+e_j$}\\
\end{array}\right.
\end{equation*}
Notice that this definition can be extended linearly to the whole of $\mathfrak u^D(q)$.

\null
\noindent Denote by $[\pm n]$ the set $\{\pm 1, \pm 2, \dots , \pm n\}$. Combinatorially, linear combinations of the matrices $y_{\alpha}$ with at most one nonzero entry in every row and column can be seen as \emph{labelled} $D_{2n}$-\emph{partitions} or simply $D_{2n}(q)$-\emph{partitions}, which consists of triples $(i,j,a)$ where $i,j\in [\pm n]$ and $a\in\mathbb F_q^*$. Any triple of this form is called a \emph{labelled arc} and will be represented as ${\small i}\larc{a} {\small j}$. Thus, we have the following definition.

\begin{definition}\label{D-partition}
A \emph{$D_{2n}(q)$-partition} $\lambda$ of $[\pm n]$ is a set of labelled arcs in $[\pm n]$ such that for $j\neq -i$:
\begin{itemize}
 \item[$(a).$] If ${\small i}\larc{a} {\small j}\in\lambda$ then ${\small -j}\larc{-a} {\small -i}\in\lambda$\\
 \item[$(b).$] If ${\small i}\larc{a} {\small j}\in\lambda$ and $i\prec k\prec j$ then ${\small i}\larc{b} {\small k}\notin\lambda,\;{\small k}\larc{b} {\small j}\notin\lambda$. 
\end{itemize} 
 We write $\lambda\in D_{2n}(q)$ to indicate that $\lambda$ is a $D_{2n}(q)$-partition.
\end{definition}
\noindent The number of $D_{2n}(q)$-partitions is given in \cite{marberg} where the notion of labelled $D$ partitions was previously defined, as well as their analog in type $B$.

For $\lambda\in D_{2n}(q)$, we define the corresponding matrix $y_{\lambda}\in\mathfrak u^D(q)$ by
\begin{equation}\label{partition_matrix}
y_{\lambda} = \displaystyle\sum_{\small{i}\larc{a}\small{j}\in\lambda}ae_{i,j}.
\end{equation}

\noindent The map $\lambda\mapsto y_{\lambda}$ defines a bijection between set partitions in $D_{2n}(q)$ and matrices in $\mathfrak u^D(q)$ with at most one non zero element in every row and column.

Every $\lambda\in D_{2n}(q)$ can be written uniquely as $\lambda = \lambda^+\cup\lambda ^-$, where:
\begin{itemize}
\item $\lambda^+\cap\lambda^-=\emptyset$
\item ${\small i}\larc{a} {\small j}\in\lambda^+$ if and only if ${\small -j}\larc{-a} {\small -i}\in\lambda^-$ where $i> 0$ and $i<| j |$
 \end{itemize}

\noindent In view of this, $\lambda$ is completely determined by $\lambda^+$ (or $\lambda^-$). Thus, every arc 
${\small i}\larc{a} {\small j}\in\lambda^+$ can be represented by the triple $\{(i,j,a)\}$. In this case, the triple $\{(-j,-i,-a)\}\in\lambda^-$.

\subsection{Superclasses and supercharacters}

In this section we describe combinatorially the superclasses and supercharacters of $U^D_{2n}(q)$ using $D_{2n}(q)$-partitions and keeping in mind that $U^D_{2n}(q)$ is a subgroup of $U_{2n}(q)$. Using algebraic varieties, Andr\'e and Neto proved that supercharacters and superclasses of the group $U_{2n}^D(q)$ are indexed by matrices in $\mathfrak u^D(q)$ with at most one nonzero element in every row and column (see \cite{andre_06}). Thus, they can be indexed using $D_{2n}(q)$-partitions as well.

The group $U_{2n}(q)$ acts on its nilpotent algebra $\mathfrak u_{2n}(q)$ by left and right multiplication. It can be shown that when adding the identity matrix $I_{2n}$ to each one of these orbits we get the superclasses of $U_{2n}(q)$ (see \cite{diaconis}). 

 Let $\lambda\in\ D_{2n}(q)$ and let $y_{\lambda}$ as in (\ref{partition_matrix}). Since $\mathfrak u^D(q)\subset \mathfrak u_{2n}(q)$ we can consider the orbit
 \begin{equation*}
 V_{\lambda}=U_{2n}(q)y_{\lambda} U_{2n}(q)\in\mathfrak u_{2n}(q).
 \end{equation*}
 
\noindent Notice that $V_{\lambda}$ is not necessarily in $\mathfrak u^D(q)$. However, since $V_{\lambda}+I_{2n}$ is a superclass in $U_{2n}(q)$ and $U^D_{2n}(q)$ is a subgroup of $U_{2n}(q)$,  then we have the following definition.
\begin{definition}
Let $V_{\lambda}$ as above. The \emph{superclass} in $U_{2n}^D(q)$ associated to $\lambda$ is denoted by $K_{\lambda}$ and is defined as $K_{\lambda}=U_{2n}^D(q)\cap (V_{\lambda}+I_{2n})$. 
\end{definition}
As mentioned in the introduction $U^D_{2n}(q)$ is not an algebra group, i.e., $U^D_{2n}(q)\neq I_{2n}+\mathfrak u_{2n}^D(q)$. Yet there is a bijective correspondence between $U^D_{2n}(q)$ and $\mathfrak u^D(q)$. This bijection is provided by the following lemma of Andr\'e and Neto.

\begin{lemma}\cite[Lemma 2.3]{andre_06}\label{bijection}
 Let $\lambda$ be a $D_{2n}$-partition and let $I$ denote the identity matrix of the corresponding size. Put $x$ and $y$ as
\begin{equation*}\label{bijection}
x=\begin{pmatrix}
P & \vline & PQ\\
\hline
0 &\vline & JP^{-t}J
\end{pmatrix}\in U^D(q)\;\;\;\;\text{ and }
\;\;\;\;\;y=\begin{pmatrix}
P-I &\vline & Q\\
\hline
0 &\vline & -J(P-I)^{t}J
\end{pmatrix}\in \mathfrak u^D(q).
\end{equation*}
Then $x\in K_{\lambda}$ if and only if $y\in V_{\lambda}$.
\end{lemma}

\vspace{.3cm}

\noindent To illustrate this lemma, let us consider the following example:

\begin{example}\label{part2} Let $n=5$ and let $\lambda$ be the $D_{2n}(q)$-partition given by

\begin{equation*}
\lambda =\begin{tikzpicture}[baseline=.2cm]
	\foreach \x in {1,2,3,4,5,6,7,8} 
		\node (\x) at (\x/2,0) [inner sep=0pt] {$\bullet$};
	\node at (1/2,-.2) {$\scriptstyle 1$};
	\node at (2/2,-.2) {$\scriptstyle 2$};
	\node at (2.9/2,-.2) {$\scriptstyle 3$};
	\node at (3.9/2,-.2) {$\scriptstyle 4$};
	\node at (4.8/2,-.2) {$\scriptstyle -4$};
	\node at (5.8/2,-.2) {$\scriptstyle -3$};
	\node at (6.8/2,-.2) {$\scriptstyle -2$};
	\node at (7.8/2,-.2) {$\scriptstyle -1$};
	\draw (1) .. controls (1.25/2,.35) and (1.75/2,.35) ..  node [above=-2.2pt] {$\scriptstyle a$} (2); 
	\draw (2) .. controls (2.25/2,.5) and (2.75/2,.5) ..  node [above=-2pt] {$\scriptstyle b$} (3);
	\draw (3) .. controls (3.75/2,.5) and (4.25/2,.5) ..  node [above=-2pt] {$\scriptstyle c$} (5);
	\draw (4) .. controls (4.75/2,.5) and (5.25/2,.5) ..  node [above=-2pt] {$\scriptstyle -c$} (6);
	\draw (6) .. controls (6.25/2,.4) and (6.75/2,.4) ..  node [above=-2pt] {$\scriptstyle -b$} (7);
	\draw (7) .. controls (7.25/2,.4) and (7.75/2,.4) ..  node [above=-2pt] {$\scriptstyle -a$} (8); 
\end{tikzpicture}
\end{equation*}

\noindent A natural representative for the orbit $V_{\lambda}$ is given by the corresponding $y_{\lambda}$. In this example we have
\begin{align*}
y_{\lambda} = \begin{pmatrix}
 0 & a & 0 & 0 &\vline & 0 & 0 & 0 & 0 \\
  & 0 & b & 0 &\vline & 0 & 0 & 0 & 0 \\
  &  & 0 & 0 &\vline & c & 0 & 0 & 0\\
  &  &  & 0 &\vline & 0 & -c &  0 & 0\\
 \hline
   &  & &  &\vline & 0 & 0 & 0 & 0\\
   &  &  &  &\vline &  & 0 & -b & 0\\
  &  &  &  &\vline &  &  & 0 & -a\\
  &  &  &  &\vline &  &  &  & 0\\
\end{pmatrix}
\;\;\;\;\;\; \text{ and }\;\;\;\;
x_{\lambda} = \begin{pmatrix}
 1 & a & 0 & 0 &\vline & 0 & 0 & 0 & 0 \\
  & 1 & b & 0 &\vline & 0 & 0 & 0 & 0 \\
  &  & 1 & 0 &\vline & c & bc & 0 & 0\\
  &  &  & 1 &\vline & 0 & -c &  0 & 0\\
 \hline
   &  &  &  &\vline & 1 & 0 & 0 & 0\\
   &  &  &  &\vline &  & 1 & -b & 0\\
  &  &  &  &\vline &  &  & 1 & -a\\
  &  &  &  &\vline &  &  &  & 1\\
\end{pmatrix}
\end{align*}
where $x_{\lambda}$ is the matrix in $K_{\lambda}$ given by the lemma \ref{bijection}.
\end{example}

Henceforth we denote by $\theta$ a fixed nontrivial homomorphism from the additive group of $\mathbb F_q$ to $\mathbb C^*$. The following theorem of Andr\'e and Neto defines the supercharacters of type $D$.

\begin{theorem}\cite[Theorem 5.3]{andre}
Let $\lambda$ be a $D_{2n}(q)$-partition and let $x_\mu$ be the superclass associated to the $D_{2n}(q)$-partition $\mu$. Then the complex valued function $\chi^{\lambda}$ which is constant on superclasses of $U^D_{2n}(q)$ is given by
	
\begin{equation}\label{supercharacter_formula}
\chi^{\lambda}(x_\mu)=
\left\{
  \begin{array}{l l}
    \dfrac{\chi^{\lambda}(1)}{q^{|\{k \stackrel \frown {} l\in\mu^+ | i\prec k\prec l\prec j , i \stackrel \frown {} j\in\lambda^+ \}|}}
  \displaystyle\large\prod_{{\small i}\larc{a} {\small j} \in\lambda^+,\; 
    {\small i}\larc{b} {\small j}\in\mu^+}
    \theta (ab) & \quad \text{if }{\small i}\larc{a} {\small k}\in\lambda \text{ and }i\prec l\prec k,\\ & \;\;\text{           then }{\small i}\larc{b} {\small l},{\small l}\larc{b} {\small k}\notin\mu ,\\ \vspace{.2cm}
  0 & \quad \text{otherwise.}\\
  \end{array} \right.
\end{equation}
The set of functions $\{\chi^{\lambda}:\lambda\in D_{2n}(q)\}$ forms a supercharacter theory for the group $U^D_{2n}(q)$.
\end{theorem}

A few remarks are worth mentioning about some of the combinatorial properties of supercharacters. For an algebraic proof see \cite{andre_06}:
\begin{itemize}
\item 
$\chi^{\lambda}(1)= 
\left\{
  \begin{array}{l l}
  q^{j-i-1} & \quad\text{ if } j\preceq n\\
  q^{2n-i-j}& \quad \text{ otherwise}\\
  \end{array} \right.
  $\;\;when $\lambda^+=\{(i,j,a)\}$ is a single arc. 
  \item $\chi^{\lambda}=\displaystyle\large\prod_{\lambda_{ij}\in\lambda}\chi^{\lambda_{ij}}$
\;\;  where $\lambda_{ij}=\{{\small{i}}\larc{a}{\small{j}},{\small{-j}}\larc{-a}{\small{-i}}\}\in\lambda$. Thus, $\chi^{\lambda}(1)=1$ if and only if $j=i+1$ for every $\lambda_{ij}$.
  \end{itemize}

\subsection{Product and coproduct}

Let $\mathbf{SC}^D_{2n}$ be the vector space of superclass functions over the group $U_{2n}^D(q)$. This is the space of functions $\alpha:U_{2n}^D(q)\rightarrow\mathbb C$ that are constant on superclasses. Now that we know how superclasses and supercharacters look like as matrices and as partitions, we will define a product and a coproduct on the graded vector space $\textbf{SC}^D=\bigoplus_{n\geq 0} \mathbf{SC}^D_{2n}$. By convention, $\mathbf{SC}^D_0 = \mathbb C$. As proved in \cite[Theorem 4.1]{andre}, the supercharacters of $U_{2n}^D(q)$ form a basis for $\mathbf{SC}^D_{2n}$. 

\noindent Given $\lambda\in D_{2n}(q)$ define $\kappa_{\lambda}\in\mathbf{SC}^D_{2n}$ as  the function with the formula
\begin{equation*}
\kappa_{\lambda}(x_{\mu})=\left\{
  \begin{array}{l l}
  1 & \quad\text{ if }x_{\mu}\text{ is in the superclass of }x_{\lambda}\\
  0 & \quad\text{ otherwise}
   \end{array} \right.
\end{equation*}
for $\mu\in D_{2n}(q)$. These \emph{superclass characteristic functions} of course form another basis of $\mathbf{SC}^D$

First, we will endow the vector space $\mathbf {SC}^D$with a coalgebra structure and thus we want to define a coproduct.     As before we denote by $[\pm n]$ the set $\{\pm 1, \pm 2, \dots , \pm n\}$ so that if $A\subseteq [n]$ then  $[\pm A] = A\cup -A$ where $-A=\{-i\;:\;i\in A\}$. We start by giving the following definition.
\begin{definition}
Let $J=(J_1|\dots |J_r)$ be a sequence of subsets of $[\pm n]$. We call $J$ a \emph{$D$-set partition} of $[\pm n]$ if for all $i\neq j$ such that $1\leq i,j\leq r$ we have that $J_i=-J_i$, $J_i\cap J_j=\emptyset$ and $\cup_iJ_i=[\pm n]$.
\end{definition}

Remark that for a $D$-set partition $J=(J_1|\cdots |J_r)$,  each part $J_i$ is completely determined by the set $J_i\cap [n]$. In particular, if $J=(J_1|J_2)$ and $J_1\cap [n]=A$ then $J_2\cap [n] = A^c$, where $A^c=[n]\setminus A$. For this reason, we will write such a $D$-set partition as $J=(A|A^c)$. As an example, with $n=4$ we have that the $D$-set partition $J=(13\bar1\bar3|24\bar2\bar4)$ can be written as $J=(A|A^c)$ where $A=\{1,3\}$. 

\vspace{.3cm}

 Let $J=(A|A^c)$ be an ordered $D$-set partition of $[\pm n]$ and let 
$$
S_J(q)=\left\{\lambda\in D_{2n}(q):{\small i}\larc{a} {\small j} \in\lambda\text{ implies }i,j\text{ are in the same part of }J\right\}
$$
and define the \emph{standardization map} by the bijection

\begin{equation}\label{standard}
\text{st}_J:S_J(q)\rightarrow S_{[\pm |A|]}(q)\times S_{[\pm |A^c|]}(q)
\end{equation}

\noindent that relabels the indices of partitions in $S_J(q)$ according to the unique order-preserving map
\begin{equation}\label{straight}
\text{st}_A:\pm A\rightarrow [\pm |A|]
\end{equation}
where $A$ is a part of $J$.

\noindent As an example, let $J=\{134|25\}$ and let $\lambda$ be given by
\begin{equation}\label{lambda}
\lambda =\begin{tikzpicture}[baseline=.2cm]
	\foreach \x in {1,2,3,4,5,6,7,8,9,10} 
		\node (\x) at (\x/2,0) [inner sep=0pt] {$\bullet$};
	\node at (1/2,-.2) {$\scriptstyle 1$};
	\node at (2/2,-.2) {$\scriptstyle 2$};
	\node at (2.9/2,-.2) {$\scriptstyle 3$};
	\node at (3.9/2,-.2) {$\scriptstyle 4$};
	\node at (4.8/2,-.2) {$\scriptstyle 5$};
	\node at (5.8/2,-.2) {$\scriptstyle -5$};
	\node at (6.8/2,-.2) {$\scriptstyle -4$};
	\node at (7.8/2,-.2) {$\scriptstyle -3$};
	\node at (8.8/2,-.2) {$\scriptstyle -2$};
	\node at (9.8/2,-.2) {$\scriptstyle -1$};
	\draw (1) .. controls (1.75/2,.35) and (2.25/2,.35) ..  node [above=-2.2pt] {$\scriptstyle a$} (3); 
	\draw (2) .. controls (3/2,.5) and (4/2,.5) ..  node [above=-2pt] {$\scriptstyle b$} (5);
	\draw (8) .. controls (8.75/2,.35) and (9.25/2,.35) ..  node [above=-2pt] {$\scriptstyle -a$} (10);
	\draw (6) .. controls (7/2,.5) and (8/2,.5) ..  node [above=-2pt] {$\scriptstyle -b$} (9);
	\draw (3) .. controls (4.25/2,.5) and (5.75/2,.5) ..  node [above=-2pt] {$\scriptstyle c$} (7);
	\draw (4) .. controls (5.25/2,.5) and (6.75/2,.5) ..  node [above=-2pt] {$\scriptstyle -c$} (8); 
\end{tikzpicture}
\end{equation}
then

\begin{equation*}\text{st}_J(\lambda)=\begin{tikzpicture}[baseline=.2cm]
	\foreach \x in {1,2,3,4,5,6} 
		\node (\x) at (\x/2,0) [inner sep=0pt] {$\bullet$};
	\node at (1/2,-.2) {$\scriptstyle 1$};
	\node at (2/2,-.2) {$\scriptstyle 2$};
	\node at (3/2,-.2) {$\scriptstyle 3$};
	\node at (3.8/2,-.2) {$\scriptstyle -3$};
	\node at (4.8/2,-.2) {$\scriptstyle -2$};
	\node at (5.8/2,-.2) {$\scriptstyle -1$};
	\draw (1) .. controls (1.25/2,.35) and (1.75/2,.35) ..  node [above=-2.2pt] {$\scriptstyle a$} (2); 
	\draw (5) .. controls (5.25/2,.35) and (5.75/2,.35) ..  node [above=-2pt] {$\scriptstyle -a$} (6);
	\draw (2) .. controls (2.25/2,.4) and (3.75/2,.4) ..  node [above=-2pt] {$\scriptstyle c$} (4);
	\draw (3) .. controls (3.25/2,.4) and (4.75/2,.4) ..  node [above=-2pt] {$\scriptstyle -c$} (5); 
\end{tikzpicture}\times 
\begin{tikzpicture}[baseline=.2cm]
	\foreach \x in {1,2,3,4} 
		\node (\x) at (\x/2,0) [inner sep=0pt] {$\bullet$};
	\node at (1/2,-.2) {$\scriptstyle 1$};
	\node at (2/2,-.2) {$\scriptstyle 2$};
	\node at (2.8/2,-.2) {$\scriptstyle -2$};
	\node at (3.8/2,-.2) {$\scriptstyle -1$};
	\draw (1) .. controls (1.25/2,.35) and (1.75/2,.35) ..  node [above=-2.2pt] {$\scriptstyle b$} (2); 
	\draw (3) .. controls (3.25/2,.35) and (3.75/2,.35) ..  node [above=-2pt] {$\scriptstyle -b$} (4); 
\end{tikzpicture}\in S_{[\pm 3]}(q)\times S_{[\pm 2]}(q) \end{equation*}

\begin{definition}
Let $J=(A|A^c)$ be an ordered $D$-set partition of $[\pm n]$. Define $U_J^D\subseteq U^D$, where $U^D=U^D_{2n}(q)$, as 
\begin{equation*}
U_J^D=\{x\in U^D:x_{ij}\neq 0 \text{ implies } i,j \text{ are in the same part of }J\}
\end{equation*}
\end{definition}

\noindent The map in (\ref{standard}) can be extended to produce an isomorphism $\text{st}_J:U^D_J\rightarrow U_{2|A|}^D(q)\times U_{2|A^c|}^D(q)$ by reordering the rows and columns as in (\ref{straight}).
%\noindent For $\mu\in S_{A}(q)$ and $\nu\in S_{A^c}(q)$ we let $x_{\mu\times\nu}$ be the corresponding superclass representative in $U_{\text{st}(J)}^D(q)$. In our current example, $\mu$ is the set of arcs in $\lambda$ on $\{1,3,4\}$ and its negatives, and $\nu$ is the complementary set of arcs. The superclass $x_{\mu\times\nu}$ is given as in example \ref{part2}, in other words, is the superclass corresponding to $x_{\text{st}^{-1}_J(\mu)}\times x_{\text{st}^{-1}_J(\nu)}$.

\null

The restriction map on \textbf{SC}$^D_{2n}(q)$ is given by
\begin{equation}\label{resmap}
{\large\text{Res}}_{\text{st}_J(U_{J}^D)}^{U^D} : \mathbf{SC}_{2n}^D(q)\rightarrow\mathbf{SC}_{2|A|}^D(q)\otimes\mathbf{SC}_{2|A^c|}^D(q)
\end{equation}
where ${\large\text{Res}}_{\text{st}_J(U_{J}^D)}^{U^D}(\chi)(u)=\chi(\text{st}_J^{-1}(u))$ for $u\in U_{2|A|}^D(q)\times U_{2|A^c|}^D(q)$.

\vspace{.3cm}
\noindent We must show that this map is well-defined, i.e., takes superclass functions to superclass functions. Since supercharacters of type $D$ are restrictions of supercharacters of type $A$ (see \cite[Proposition 2.2]{andre_06}), equation (\ref{resmap}) can be written as 
$$
{\large\text{Res}}_{\text{st}_J(U_{J}^D)}^{U^D}(\zeta_{U^D})(u)=\zeta_{U^D}(\text{st}_J^{-1}(u))
$$
for some $\zeta$ supercharacter of type $A$ such that its restriction $\zeta_{U^D}$ to $U^D$ is precisely $\chi$. Also, in \cite[Theorem 6.4]{diaconis}, Diaconis and Isaacs prove that superclass functions of type $A$ restrict to superclass functions. Putting these facts together we conclude that the restriction map sends superclass functions of type $D$ to superclass functions.

\vspace{.3cm}
\noindent Now, define the coproduct on supercharacters as 

$$
\Delta(\chi) :=\displaystyle\sum_{{J}}{\large\text{Res}}_{\text{st}_J(U_{J}^D)}^{U^D}(\chi)
$$
summing over all ordered $D$-set partitions $J=(A|A^c)$ of $[\pm n]$. 

\vspace{.3cm}

Given a subset $A$ of $[n]$ and $\lambda\in D_{2n}(q)$ let $\lambda|_A$ denote the restriction of $\lambda$ to the set $[\pm A]$. This is, the ground set of $\lambda|_A$ is $[\pm A]$ and $\small{i}\larc{a}\small{j}\in\lambda|_A$ if $\small{i}\larc{a}\small{j}\in\lambda$. Now, denoting by $\mathcal A(\lambda)$ the set of arcs of $\lambda$ we see that $\mathcal A(\lambda|_A)\sqcup\mathcal A(\lambda|_{A^c})\subseteq \mathcal A(\lambda)$. When the equality $\mathcal A(\lambda|_A)\cup\mathcal A(\lambda|_{A^c})= \mathcal A(\lambda)$ holds we write $\lambda = \lambda|_A\sqcup \lambda|_{A^c}$.

The following proposition tells us how to compute the coproduct in the superclass characteristic functions.

\begin{proposition}\label{coproduct}
Let $\lambda$ be a $D_{2n}(q)$-partition. Then
\begin{equation*}
\Delta (\kappa_{\lambda})=\displaystyle\sum_{\substack{\lambda =\lambda |_A\sqcup\lambda |_{A^c}}}\kappa_{\text{st}_A(\lambda |_A )}\otimes\kappa_{\text{st}_{A^c}(\lambda |_{A^c})}
\end{equation*}
summing over all $A\subseteq [n]$ such that $\lambda = \lambda |_A\sqcup\lambda |_{A^c}$.
\end{proposition}

\begin{proof}
For an ordered $D$-set composition $J=(A|A^c)$ we have that $st_A(\lambda|_A)\in D_{2|A|}(q)$ and $st_{A^c}(\lambda|_{A^c})\in D_{2|A^c|}(q)$. Given $\mu\in D_{2|A|}(q)$ and $\nu\in D_{2|A^c|}(q)$ denote by $x_{\mu\times\nu}$ the natural orbit superclass representative indexed by $\mu\times\lambda$ in the group $U_{2|A|}^D(q)\times U_{2|A^c|}^D(q)$. Then we have
\begin{equation*}
{\large\text{Res}}_{\text{st}_J(U_{J}^D)}^{U^D}(\kappa_{\lambda})(x_{\mu\times\nu}) =
\left\{
  \begin{array}{l l}
1 & \text{if }\lambda|_A = \text{st}^{-1}_J(\mu) \text{ and }\lambda |_{A^c} = \text{st}^{-1}_J(\nu)\\
0 & \text{ otherwise}
\end{array}\right.\\
\end{equation*}
This means that ${\large\text{Res}}_{\text{st}_J(U_{J}^D)}^{U^D}(\kappa_{\lambda})(x_{\mu\times\nu})\neq 0$ when $\text{st}_J(\lambda)=\mu\times\nu$. This concludes the proof.
\end{proof}

\begin{example}
Let $\lambda\in D_{12}(q)$ given by
\begin{equation*}
\lambda: \begin{tikzpicture}[baseline=.2cm]
	\foreach \x in {1,2,3,4,5,6,7,8,9,10,11,12} 
		\node (\x) at (\x/2,0) [inner sep=0pt] {$\bullet$};
	\node at (1/2,-.2) {$\scriptstyle 1$};
	\node at (2/2,-.2) {$\scriptstyle 2$};
	\node at (3/2,-.2) {$\scriptstyle 3$};
	\node at (4/2,-.2) {$\scriptstyle 4$};
	\node at (5/2,-.2) {$\scriptstyle 5$};
	\node at (6/2,-.2) {$\scriptstyle 6$};
	\node at (6.8/2,-.2) {$\scriptstyle -6$};
	\node at (7.8/2,-.2) {$\scriptstyle -5$};
	\node at (8.8/2,-.2) {$\scriptstyle -4$};
	\node at (9.8/2,-.2) {$\scriptstyle -3$};
	\node at (10.8/2,-.2) {$\scriptstyle -2$};
	\node at (11.8/2,-.2) {$\scriptstyle -1$};
	\draw (1) .. controls (1.75/2,.5) and (3.25/2,.5) ..  node [above=-2.2pt] {$\scriptstyle a$} (4); 
	\draw (4) .. controls (4.75/2,.5) and (6.25/2,.5) ..  node [above=-2pt] {$\scriptstyle b$} (7); 
	\draw (3) .. controls (3.25/2,.4) and (4.75/2,.4) ..  node [above=-2.2pt] {$\scriptstyle c$} (5); 
	\draw (8) .. controls (8.75/2,.4) and (9.25/2,.4) ..  node [above=-2.2pt] {$\scriptstyle -c$} (10); 
	\draw (6) .. controls (6.75/2,.5) and (8.25/2,.5) ..  node [above=-2pt] {$\scriptstyle -b$} (9); 
	\draw (9) .. controls (9.75/2,.5) and (11.25/2,.5) ..  node [above=-2.2pt] {$\scriptstyle -a$} (12); 
\end{tikzpicture}
\end{equation*}

\noindent Then,

\begin{align*}
\Delta(\kappa_{\lambda}) &=\kappa_{\lambda}\otimes\kappa_{\emptyset}+\kappa_{\begin{tikzpicture}[baseline=.2cm]
	\foreach \x in {1,2,3,4,5,6} 
		\node (\x) at (\x/4,0) [inner sep=0pt] {$\scriptstyle\bullet$};
	\node at (1/4,-.2) {$\scriptscriptstyle 1$};
	\node at (2/4,-.2) {$\scriptscriptstyle 2$};
	\node at (2.8/4,-.2) {$\scriptscriptstyle 3$};
	\node at (3.8/4,-.2) {$\scriptscriptstyle -3$};
	\node at (4.8/4,-.2) {$\scriptscriptstyle -2$};
	\node at (5.8/4,-.2) {$\scriptscriptstyle -1$};
	\draw (1) .. controls (1.25/4,.35) and (1.75/4,.35) ..  node [above=-2.2pt] {$\scriptscriptstyle a$} (2); 
	\draw (2) .. controls (2.25/4,.35) and (3.75/4,.35) ..  node [above=-2pt] {$\scriptscriptstyle b$} (4); 
	\draw (3) .. controls (3.25/4,.35) and (4.75/4,.35) ..  node [above=-2.2pt] {$\scriptscriptstyle -b$} (5); 
	\draw (5) .. controls (5.25/4,.35) and (5.75/4,.35) ..  node [above=-2pt] {$\scriptscriptstyle -a$} (6); 
\end{tikzpicture}}
\otimes
\kappa_{\begin{tikzpicture}[baseline=.2cm]
	\foreach \x in {1,2,3,4,5,6} 
		\node (\x) at (\x/4,0) [inner sep=0pt] {$\scriptstyle\bullet$};
	\node at (1/4,-.2) {$\scriptscriptstyle 1$};
	\node at (2/4,-.2) {$\scriptscriptstyle 2$};
	\node at (3/4,-.2) {$\scriptscriptstyle 3$};
	\node at (3.8/4,-.2) {$\scriptscriptstyle -3$};
	\node at (4.8/4,-.2) {$\scriptscriptstyle -2$};
	\node at (5.8/4,-.2) {$\scriptscriptstyle -1$};
	\draw (2) .. controls (2.25/4,.35) and (2.75/4,.35) ..  node [above=-2.2pt] {$\scriptscriptstyle c$} (3); 
	\draw (4) .. controls (4.25/4,.35) and (4.75/4,.35) ..  node [above=-2pt] {$\scriptscriptstyle -c$} (5); 
\end{tikzpicture}}+\kappa_{\begin{tikzpicture}[baseline=.2cm]
	\foreach \x in {1,2,3,4,5,6} 
		\node (\x) at (\x/4,0) [inner sep=0pt] {$\scriptstyle\bullet$};
	\node at (1/4,-.2) {$\scriptscriptstyle 1$};
	\node at (2/4,-.2) {$\scriptscriptstyle 2$};
	\node at (3/4,-.2) {$\scriptscriptstyle 3$};
	\node at (3.8/4,-.2) {$\scriptscriptstyle -3$};
	\node at (4.8/4,-.2) {$\scriptscriptstyle -2$};
	\node at (5.8/4,-.2) {$\scriptscriptstyle -1$};
	\draw (2) .. controls (2.25/4,.35) and (2.75/4,.35) ..  node [above=-2.2pt] {$\scriptscriptstyle c$} (3); 
	\draw (4) .. controls (4.25/4,.35) and (4.75/4,.35) ..  node [above=-2pt] {$\scriptscriptstyle -c$} (5); 
\end{tikzpicture}}\otimes
\kappa_{\begin{tikzpicture}[baseline=.2cm]
	\foreach \x in {1,2,3,4,5,6} 
		\node (\x) at (\x/4,0) [inner sep=0pt] {$\scriptstyle\bullet$};
	\node at (1/4,-.2) {$\scriptscriptstyle 1$};
	\node at (2/4,-.2) {$\scriptscriptstyle 2$};
	\node at (2.8/4,-.2) {$\scriptscriptstyle 3$};
	\node at (3.8/4,-.2) {$\scriptscriptstyle -3$};
	\node at (4.8/4,-.2) {$\scriptscriptstyle -2$};
	\node at (5.8/4,-.2) {$\scriptscriptstyle -1$};
	\draw (1) .. controls (1.25/4,.35) and (1.75/4,.35) ..  node [above=-2.2pt] {$\scriptscriptstyle a$} (2); 
	\draw (2) .. controls (2.25/4,.35) and (3.75/4,.35) ..  node [above=-2pt] {$\scriptscriptstyle b$} (4); 
	\draw (3) .. controls (3.25/4,.35) and (4.75/4,.35) ..  node [above=-2.2pt] {$\scriptscriptstyle -b$} (5); 
	\draw (5) .. controls (5.25/4,.35) and (5.75/4,.35) ..  node [above=-2pt] {$\scriptscriptstyle -a$} (6); 
\end{tikzpicture}}\\ &+
\kappa_{\begin{tikzpicture}[baseline=.2cm]
	\foreach \x in {1,2,3,4,5,6,7,8} 
		\node (\x) at (\x/4,0) [inner sep=0pt] {$\scriptstyle\bullet$};
	\node at (1/4,-.2) {$\scriptscriptstyle 1$};
	\node at (2/4,-.2) {$\scriptscriptstyle 2$};
	\node at (3/4,-.2) {$\scriptscriptstyle 3$};
	\node at (4/4,-.2) {$\scriptscriptstyle 4$};
	\node at (4.8/4,-.2) {$\scriptscriptstyle -4$};
	\node at (5.8/4,-.2) {$\scriptscriptstyle -3$};
	\node at (6.8/4,-.2) {$\scriptscriptstyle -2$};
	\node at (7.8/4,-.2) {$\scriptscriptstyle -1$};
	\draw (1) .. controls (1.75/4,.35) and (2.25/4,.35) ..  node [above=-2.2pt] {$\scriptscriptstyle a$} (3); 
	\draw (3) .. controls (3.75/4,.35) and (4.25/4,.35) ..  node [above=-2pt] {$\scriptscriptstyle b$} (5); 
	\draw (4) .. controls (4.75/4,.35) and (5.25/4,.35) ..  node [above=-2.2pt] {$\scriptscriptstyle -b$} (6); 
	\draw (6) .. controls (6.75/4,.35) and (7.25/4,.35) ..  node [above=-2pt] {$\scriptscriptstyle -a$} (8); 
\end{tikzpicture}
}\otimes
\kappa_{\begin{tikzpicture}[baseline=.2cm]
	\foreach \x in {1,2,3,4} 
		\node (\x) at (\x/4,0) [inner sep=0pt] {$\scriptstyle\bullet$};
	\node at (1/4,-.2) {$\scriptscriptstyle 1$};
	\node at (2/4,-.2) {$\scriptscriptstyle 2$};
	\node at (2.8/4,-.2) {$\scriptscriptstyle -2$};
	\node at (3.8/4,-.2) {$\scriptscriptstyle -1$};
	\draw (1) .. controls (1.25/4,.35) and (1.75/4,.35) ..  node [above=-2.2pt] {$\scriptscriptstyle c$} (2); 
	\draw (3) .. controls (3.25/4,.35) and (3.75/4,.35) ..  node [above=-2pt] {$\scriptscriptstyle -c$} (4); 
\end{tikzpicture}}+
\kappa_{\begin{tikzpicture}[baseline=.2cm]
	\foreach \x in {1,2,3,4} 
		\node (\x) at (\x/4,0) [inner sep=0pt] {$\scriptstyle\bullet$};
	\node at (1/4,-.2) {$\scriptscriptstyle 1$};
	\node at (2/4,-.2) {$\scriptscriptstyle 2$};
	\node at (2.8/4,-.2) {$\scriptscriptstyle -2$};
	\node at (3.8/4,-.2) {$\scriptscriptstyle -1$};
	\draw (1) .. controls (1.25/4,.35) and (1.75/4,.35) ..  node [above=-2.2pt] {$\scriptscriptstyle c$} (2); 
	\draw (3) .. controls (3.25/4,.35) and (3.75/4,.35) ..  node [above=-2pt] {$\scriptscriptstyle -c$} (4); 
\end{tikzpicture}}\otimes
\kappa_{\begin{tikzpicture}[baseline=.2cm]
	\foreach \x in {1,2,3,4,5,6,7,8} 
		\node (\x) at (\x/4,0) [inner sep=0pt] {$\scriptstyle\bullet$};
	\node at (1/4,-.2) {$\scriptscriptstyle 1$};
	\node at (2/4,-.2) {$\scriptscriptstyle 2$};
	\node at (3/4,-.2) {$\scriptscriptstyle 3$};
	\node at (4/4,-.2) {$\scriptscriptstyle 4$};
	\node at (4.8/4,-.2) {$\scriptscriptstyle -4$};
	\node at (5.8/4,-.2) {$\scriptscriptstyle -3$};
	\node at (6.8/4,-.2) {$\scriptscriptstyle -2$};
	\node at (7.8/4,-.2) {$\scriptscriptstyle -1$};
	\draw (1) .. controls (1.75/4,.35) and (2.25/4,.35) ..  node [above=-2.2pt] {$\scriptscriptstyle a$} (3); 
	\draw (3) .. controls (3.75/4,.35) and (4.25/4,.35) ..  node [above=-2pt] {$\scriptscriptstyle b$} (5); 
	\draw (4) .. controls (4.75/4,.35) and (5.25/4,.35) ..  node [above=-2.2pt] {$\scriptscriptstyle -b$} (6); 
	\draw (6) .. controls (6.75/4,.35) and (7.25/4,.35) ..  node [above=-2pt] {$\scriptscriptstyle -a$} (8); 
\end{tikzpicture}}\\
&+ \kappa_{\begin{tikzpicture}[baseline=.2cm]
	\foreach \x in {1,2} 
		\node (\x) at (\x/4,0) [inner sep=0pt] {$\scriptstyle\bullet$};
	\node at (1/4,-.2) {$\scriptscriptstyle 1$};
	\node at (2/4,-.2) {$\scriptscriptstyle -1$};
\end{tikzpicture}}\otimes
\kappa_{\begin{tikzpicture}[baseline=.2cm]
	\foreach \x in {1,2,3,4,5,6,7,8,9,10} 
		\node (\x) at (\x/4,0) [inner sep=0pt] {$\scriptstyle\bullet$};
	\node at (1/4,-.2) {$\scriptscriptstyle 1$};
	\node at (2/4,-.2) {$\scriptscriptstyle 2$};
	\node at (3/4,-.2) {$\scriptscriptstyle 3$};
	\node at (4/4,-.2) {$\scriptscriptstyle 4$};
	\node at (5/4,-.2) {$\scriptscriptstyle 5$};
	\node at (5.8/4,-.2) {$\scriptscriptstyle -5$};
	\node at (6.8/4,-.2) {$\scriptscriptstyle -4$};
	\node at (7.8/4,-.2) {$\scriptscriptstyle -3$};
	\node at (8.8/4,-.2) {$\scriptscriptstyle -2$};
	\node at (9.8/4,-.2) {$\scriptscriptstyle -1$};
	\draw (1) .. controls (1.75/4,.35) and (2.25/4,.35) ..  node [above=-2.2pt] {$\scriptscriptstyle a$} (3); 
	\draw (3) .. controls (4.25/4,.35) and (4.75/4,.35) ..  node [above=-2pt] {$\scriptscriptstyle b$} (6);
	\draw (2) .. controls (2.75/4,.35) and (3.25/4,.35) ..  node [above=-2.2pt] {$\scriptscriptstyle c$} (4); 
	\draw (7) .. controls (7.75/4,.35) and (8.25/4,.35) ..  node [above=-2pt] {$\scriptscriptstyle -c$} (9); 
	\draw (5) .. controls (6.25/4,.35) and (6.75/4,.35) ..  node [above=-2.2pt] {$\scriptscriptstyle -b$} (8); 
	\draw (8) .. controls (8.75/4,.35) and (9.25/4,.35) ..  node [above=-2pt] {$\scriptscriptstyle -a$} (10); 
\end{tikzpicture}}
+\kappa_{\begin{tikzpicture}[baseline=.2cm]
	\foreach \x in {1,2,3,4,5,6,7,8,9,10} 
		\node (\x) at (\x/4,0) [inner sep=0pt] {$\scriptstyle\bullet$};
	\node at (1/4,-.2) {$\scriptscriptstyle 1$};
	\node at (2/4,-.2) {$\scriptscriptstyle 2$};
	\node at (3/4,-.2) {$\scriptscriptstyle 3$};
	\node at (4/4,-.2) {$\scriptscriptstyle 4$};
	\node at (5/4,-.2) {$\scriptscriptstyle 5$};
	\node at (5.8/4,-.2) {$\scriptscriptstyle -5$};
	\node at (6.8/4,-.2) {$\scriptscriptstyle -4$};
	\node at (7.8/4,-.2) {$\scriptscriptstyle -3$};
	\node at (8.8/4,-.2) {$\scriptscriptstyle -2$};
	\node at (9.8/4,-.2) {$\scriptscriptstyle -1$};
	\draw (1) .. controls (1.75/4,.35) and (2.25/4,.35) ..  node [above=-2.2pt] {$\scriptscriptstyle a$} (3); 
	\draw (3) .. controls (4.25/4,.35) and (4.75/4,.35) ..  node [above=-2pt] {$\scriptscriptstyle b$} (6);
	\draw (2) .. controls (2.75/4,.35) and (3.25/4,.35) ..  node [above=-2.2pt] {$\scriptscriptstyle c$} (4); 
	\draw (7) .. controls (7.75/4,.35) and (8.25/4,.35) ..  node [above=-2pt] {$\scriptscriptstyle -c$} (9); 
	\draw (5) .. controls (6.25/4,.35) and (6.75/4,.35) ..  node [above=-2.2pt] {$\scriptscriptstyle -b$} (8); 
	\draw (8) .. controls (8.75/4,.35) and (9.25/4,.35) ..  node [above=-2pt] {$\scriptscriptstyle -a$} (10); 
\end{tikzpicture}}
\otimes \kappa_{\begin{tikzpicture}[baseline=.2cm]
	\foreach \x in {1,2} 
		\node (\x) at (\x/4,0) [inner sep=0pt] {$\scriptstyle\bullet$};
	\node at (1/4,-.2) {$\scriptscriptstyle 1$};
	\node at (2/4,-.2) {$\scriptscriptstyle -1$};
\end{tikzpicture}}
+\kappa_{\emptyset}\otimes\kappa_{\lambda}
\end{align*}
\end{example}

\noindent It is not hard to see that the coproduct is coassociative. Also, notice that some of the beauty of this coalgebra structure is that it is directly connected to representation theory, as is the case in type $A$. 
We now define a multiplication in the space \textbf{SC}$^D$ as follows.

\begin{definition}\label{product}
For $\lambda,\mu$ labelled $D$-partitions of $[\pm k],[\pm (n-k)]$, respectively, define
$$
\kappa_{\lambda}\cdot\kappa_{\mu} :=\displaystyle\sum_{\substack{\nu \in D_{2n}(q)
}
}\kappa_{\nu}
$$
summing over all $D_{2n}(q)$-partitions $\nu$ such that $\nu|_{[\pm k]}=\lambda$ and $\nu|_{[\pm(k+1,\dots,n)]}=\mu\uparrow^k$
%\begin{itemize}
% \item ${\small i}\larc{a} {\small j} \in\gamma\Rightarrow i\preceq k\prec j$ \; and/or 
%\item ${\small i}\larc{a} {\small -j} \in\gamma\Rightarrow -j\prec -k\preceq -i$
%\end{itemize} for $1\leq i,j\leq n$.

\noindent where, for $1\leq i,j\leq n-k$
\begin{equation*}
\mu\uparrow^k =\left\{{\small (k+i)}\larc{a} {\small (k+j)}  \;:\;{\small i}\larc{a} {\small j}\in \mu\right\}\cup\left\{  {\small (k+i)}\larc{a} {\small (-k-j)} \;:\;{\small i}\larc{a} {\small -j}\in \mu\right\}
\end{equation*}
\end{definition}

\begin{example}
Denote $-i$ by $\overline i$, then
\begin{align*}
\kappa \begin{tikzpicture}[baseline=.2cm]
	\foreach \x in {1,2,3,4} 
		\node (\x) at (\x/4,0) [inner sep=0pt] {$\scriptstyle \bullet$};
	\node at (1/4,-.2) {$\scriptscriptstyle 1$};
	\node at (2/4,-.2) {$\scriptscriptstyle 2$};
	\node at (2.9/4,-.2) {$\scriptscriptstyle {\overline 2}$};
	\node at (3.9/4,-.2) {$\scriptscriptstyle {\overline 1}$};
	\draw (1) .. controls (1.25/4,.35) and (1.75/4,.35) ..  node [above=-2.2pt] {$\scriptscriptstyle a$} (2); 
	\draw (3) .. controls (3.25/4,.35) and (3.75/4,.35) ..  node [above=-2pt] {$\scriptscriptstyle\overline a$} (4); 
\end{tikzpicture}
\cdot\kappa 
	\begin{tikzpicture}[baseline=.2cm]
	\foreach \x in {1,2,3,4} 
		\node (\x) at (\x/4,0) [inner sep=0pt] {$\scriptstyle\bullet$};
	\node at (1/4,-.2) {$\scriptscriptstyle 1$};
	\node at (2/4,-.2) {$\scriptscriptstyle 2$};
	\node at (2.9/4,-.2) {$\scriptscriptstyle \overline 2$};
	\node at (3.9/4,-.2) {$\scriptscriptstyle \overline 1$};
	\draw (1) .. controls (1.75/4,.35) and (2.25/4,.35) ..  node [above=-2.2pt] {$\scriptscriptstyle b$} (3); 
	\draw (2) .. controls (2.75/4,.35) and (3.35/4,.35) ..  node [above=-2pt] {$\scriptscriptstyle\overline b$} (4); 
\end{tikzpicture} &= 
\kappa \begin{tikzpicture}[baseline=.2cm]
	\foreach \x in {1,2,3,4,5,6,7,8} 
		\node (\x) at (\x/4,0) [inner sep=0pt] {$\scriptstyle\bullet$};
	\node at (1/4,-.2) {$\scriptscriptstyle 1$};
	\node at (2/4,-.2) {$\scriptscriptstyle 2$};
	\node at (2.9/4,-.2) {$\scriptscriptstyle 3$};
	\node at (3.9/4,-.2) {$\scriptscriptstyle 4$};
	\node at (4.8/4,-.2) {$\scriptscriptstyle \overline 4$};
	\node at (5.8/4,-.2) {$\scriptscriptstyle\overline 3$};
	\node at (6.8/4,-.2) {$\scriptscriptstyle\overline 2$};
	\node at (7.8/4,-.2) {$\scriptscriptstyle\overline 1$};
	\draw (1) .. controls (1.25/4,.35) and (1.75/4,.35) ..  node [above=-2.2pt] {$\scriptscriptstyle a$} (2); 
	\draw (3) .. controls (3.5/4,.35) and (4.5/4,.35) ..  node [above=-2pt] {$\scriptscriptstyle b$} (5); 
	\draw (4) .. controls (4.5/4,.35) and (5.5/4,.35) ..  node [above=-2.2pt] {$\scriptscriptstyle\overline b$} (6); 
	\draw (7) .. controls (7.25/4,.35) and (7.75/4,.35) ..  node [above=-2pt] {$\scriptscriptstyle\overline a$} (8); 
\end{tikzpicture} + \displaystyle\sum_{c\in\mathbb F_q^*}
\kappa\begin{tikzpicture}[baseline=.2cm]
	\foreach \x in {1,2,3,4,5,6,7,8} 
		\node (\x) at (\x/4,0) [inner sep=0pt] {$\scriptstyle\bullet$};
	\node at (1/4,-.2) {$\scriptscriptstyle 1$};
	\node at (2/4,-.2) {$\scriptscriptstyle 2$};
	\node at (2.9/4,-.2) {$\scriptscriptstyle 3$};
	\node at (3.9/4,-.2) {$\scriptscriptstyle 4$};
	\node at (4.8/4,-.2) {$\scriptscriptstyle\overline 4$};
	\node at (5.8/4,-.2) {$\scriptscriptstyle\overline 3$};
	\node at (6.8/4,-.2) {$\scriptscriptstyle\overline 2$};
	\node at (7.8/4,-.2) {$\scriptscriptstyle \overline 1$};
	\draw (1) .. controls (1.25/4,.35) and (1.75/4,.35) ..  node [above=-2.2pt] {$\scriptscriptstyle a$} (2); 
	\draw [densely dotted](2) .. controls (2.25/4,.35) and (2.75/4,.35) ..  node [above=-2pt] {$\scriptscriptstyle c$} (3);
	\draw (3) .. controls (3.5/4,.35) and (4.5/4,.35) ..  node [above=-2pt] {$\scriptscriptstyle b$} (5); 
	\draw (4) .. controls (4.5/4,.35) and (5.5/4,.35) ..  node [above=-2.2pt] {$\scriptscriptstyle\overline b$} (6); 
	\draw [densely dotted](6) .. controls (6.25/4,.35) and (6.75/4,.35) ..  node [above=-2pt] {$\scriptscriptstyle\overline c$} (7);
	\draw (7) .. controls (7.25/4,.35) and (7.75/4,.35) ..  node [above=-2pt] {$\scriptscriptstyle\overline a$} (8); 
\end{tikzpicture}\\
&+ \displaystyle\sum_{c\in\mathbb F_q^*}
{\kappa}\begin{tikzpicture}[baseline=.2cm]
	\foreach \x in {1,2,3,4,5,6,7,8} 
		\node (\x) at (\x/4,0) [inner sep=0pt] {$\scriptstyle\bullet$};
	\node at (1/4,-.2) {$\scriptscriptstyle 1$};
	\node at (2/4,-.2) {$\scriptscriptstyle 2$};
	\node at (2.9/4,-.2) {$\scriptscriptstyle 3$};
	\node at (3.9/4,-.2) {$\scriptscriptstyle 4$};
	\node at (4.8/4,-.2) {$\scriptscriptstyle\overline 4$};
	\node at (5.8/4,-.2) {$\scriptscriptstyle \overline 3$};
	\node at (6.8/4,-.2) {$\scriptscriptstyle \overline 2$};
	\node at (7.8/4,-.2) {$\scriptscriptstyle \overline 1$};
	\draw (1) .. controls (1.25/4,.35) and (1.75/4,.35) ..  node [above=-2.2pt] {$\scriptscriptstyle a$} (2); 
	\draw [densely dotted](2) .. controls (2.5/4,.35) and (3.5/4,.35) ..  node [above=-2pt] {$\scriptscriptstyle c$} (4);
	\draw (3) .. controls (3.5/4,.35) and (4.5/4,.35) ..  node [above=-2pt] {$\scriptscriptstyle b$} (5); 
	\draw (4) .. controls (4.5/4,.35) and (5.5/4,.35) ..  node [above=-2.2pt] {$\scriptscriptstyle \overline b$} (6); 
	\draw [densely dotted](5) .. controls (5.5/4,.35) and (6.5/4,.35) ..  node [above=-2pt] {$\scriptscriptstyle \overline c$} (7);
	\draw (7) .. controls (7.25/4,.35) and (7.75/4,.35) ..  node [above=-2pt] {$\scriptscriptstyle \overline a$} (8); 
\end{tikzpicture}
\end{align*}
\end{example}

Since we want to induce an algebra structure on \textbf{SC}$^D$, we need to prove that Definition \ref{product} is indeed a product in the sense that it should be associative. This will be shown once we introduce the $P$-basis. In order to motivate somehow this definition of the product, besides being a ``natural" way of doing it, in type $A$ the product structure is raised from the \emph{inflation} map on superclass functions of that type. When trying to obtain the product from representation theory in type $D$, the analogous inflation map in this case failed in the sense that superclass functions are not mapped to superclass functions anymore. For this reason, instead of deducing Definition \ref{product} from a representation-theory point of view, the product was directly defined in this way. Nevertheless, in Proposition \ref{superchar_product}, we will see that the connection with representation theory remains strong. As a final remark, before proving the main result of this paper, this product differs from the one defined for $\textbf{SC}$ in \cite{aim}. The difference is that here we do not concatenate $\lambda$ and $\mu$. Instead, we put $\mu$ in between $\lambda^+$ and $\lambda^-$. This resembles the product defined in \cite[Section 3.5] {aaron} for the Hopf monoid \textbf{Pal} of palindromic set compositions. In Section \ref{future_work} we point out that this supercharacter theory of type $D$, in particular, carries a Hopf monoid structure. Yet this Hopf monoid is different from the Hopf monoid \textbf{Pal}, since their coproduct structures are different. We will give a brief description of the Hopf monoid $\mathbf{SC}^D$ in section \ref{future_work}.

Next, we define a different basis for $\mathbf{SC}$, in order to make computations easier. 
\begin{definition}
Let $\lambda, \mu$ be $D_{2n}(q)$-partitions. We say that $\lambda\leq\mu$ if $\mathcal A(\lambda)\subseteq\mathcal A(\mu)$ where $\mathcal A(\lambda)$ denotes the set of arcs in $\lambda$.\end{definition}

Given $\lambda$, we denote by $P_{\lambda}$ the superclass function defined as
\begin{equation*}
P_{\lambda} := \displaystyle\sum_{\mu\geq\lambda}\kappa_{\mu}.
\end{equation*}
From here, we see that $\{P_{\lambda}\}_{\lambda\in D_{2n}(q)}$ forms a basis for $\textbf{SC}^D$ as $n\geq 0$. This basis is called $P$-\emph{basis}.

\begin{proposition}\label{pbasis} The $P$-basis multiplies and comultiplies as follows:

\begin{itemize}
\item [$(a)$] For $\mu , \nu$ labelled $D$-partitions of $[\pm k],[\pm (n-k)]$, respectively,  we have 
\begin{equation}\label{prod_p}
P_{\mu}\cdot P_{\nu} = P_{\mu\sqcup \nu\uparrow^k}
\end{equation}

\item [$(b)$] For $\lambda\in D_{2n}(q)$ we have 
\begin{equation}\label{coprod_p}
\Delta (P_{\lambda})=\displaystyle\sum P_{\text{st}_A(\mu )}\otimes P_{\text{st}_{A^c}(\nu)}
\end{equation}
summing over all subsets $A\subseteq [n]$ such that $\lambda = \mu\sqcup\nu$ where  $\lambda |_A=\mu$ and  $\lambda |_{A^c}=\nu$.
\end{itemize}

\end{proposition}

\begin{proof}
$(a).$ The left hand side of (\ref{prod_p}) gives us
\begin{equation*}
P_{\mu}\cdot P_{\nu} = \left(\displaystyle\sum_{\sigma\geq\mu}\kappa_{\sigma}\right)\cdot\left(\displaystyle\sum_{\delta\geq\nu}\kappa_{\delta}\right) = \displaystyle\sum_{\sigma\geq\mu}\displaystyle\sum_{\delta\geq\nu}\kappa_{\sigma}\cdot\kappa_{\delta}\\
\end{equation*}
Notice that the minimal element in this last equality corresponds to $\kappa_{\mu\sqcup \;\nu\uparrow^ k}$, where $\sqcup$ stands for disjoint union, and every other term in each $\kappa_{\sigma}\cdot\kappa_{\delta}$ is associated to a $D_{2n}(q)$-partition $\tau$ such that $\tau >\mu\sqcup \nu\uparrow^k$. %This is due to the fact that each term in $\kappa_{\sigma}\cdot\kappa_{\delta}$ has multiplicity one and they do not appear in any other of the products $\kappa_{\sigma'}\cdot\kappa_{\delta'}$.
On the other hand
\begin{align*}
P_{\mu\sqcup \nu\uparrow^k} &= \displaystyle\sum_{\tau\geq\mu\;\sqcup\; \nu\uparrow^k}\kappa_{\tau}.
\end{align*}
This concludes part $(a)$.
\noindent To prove part $(b)$ notice that from the left hand side of (\ref{coprod_p}) we have
\begin{align*}
\Delta (P_{\lambda}) &= \Delta\left(\displaystyle\sum_{\delta\geq\lambda}\kappa_{\delta}\right)
%= \displaystyle\sum_{\delta\geq\lambda}\Delta(\kappa_{\delta})\\
= \displaystyle\sum_{\delta\geq\lambda}\displaystyle\sum_{\substack{{A\subseteq [n]}\\ \delta=\tau\sqcup\sigma}}\kappa_{\text{st}_A(\tau)}\otimes\kappa_{\text{st}_{A^c}(\sigma)}
\end{align*}
\noindent On the other hand, the right hand side of (\ref{coprod_p}) gives us
\begin{align*}
\displaystyle\sum_{\substack{\lambda =\mu\sqcup\nu}}P_{\text{st}_A(\mu )}\otimes P_{\text{st}_{A^c}(\nu)} &=%\displaystyle\sum_{\substack{\lambda =\mu\sqcup\nu}}\left( \displaystyle\sum_{\tau\geq\mu}\kappa_{\text{st}_A(\tau)}\otimes\displaystyle\sum_{\sigma\geq \nu}\kappa_{\text{st}_{A^c}(\sigma)}  \right)\\
\displaystyle\sum_{\substack{\lambda =\mu\sqcup\nu}}\displaystyle\sum_{\tau\geq\mu}\displaystyle\sum_{\sigma\geq \nu} \kappa_{\text{st}_A(\tau)}\otimes\kappa_{\text{st}_{A^c}(\sigma)}
\end{align*}

\noindent Now, every $\delta\geq\lambda$ such that $\delta = \tau\sqcup\sigma$ is such that $\lambda = (\tau\cap\lambda)\sqcup(\sigma\cap\lambda)$. This last decomposition of $\lambda$ can be written as $\lambda = \mu\sqcup\nu$ and we see that $\tau\geq\mu$, $\sigma\geq\nu$. Similarly, If $\lambda=\mu\sqcup\nu$ then $\delta=\tau\cup\sigma$ is such that $\delta\geq\lambda$, for $\tau\geq\mu,\sigma\geq\nu$.

\noindent The proposition follows.

\end{proof}

Notice that by the simplicity of the multiplication in the $P$-basis, we see that the definition \ref{product} gives an associative operation. Indeed, for $\lambda,\nu,\mu$ labelled $D$-partitions of $[\pm k],[\pm l],[\pm m]$, respectively, we have
\begin{align*}
(P_{\lambda}\cdot P_{\mu})\cdot P_{\nu} &= P_{\lambda\;\sqcup\;\mu\uparrow^k}\cdot P_{\nu}\\
&= P_{\lambda\;\sqcup\;\mu\uparrow^k\;\sqcup\;\nu\uparrow^{k+l}}\\
&= P_{\lambda\;\sqcup\;(\mu\;\sqcup\;\nu\uparrow^{l})\uparrow^k}\\
&= P_{\lambda}\cdot P_{\mu\;\sqcup\;\nu\uparrow^{l}} = P_{\lambda}\cdot (P_{\mu}\cdot P_{\nu}).
\end{align*}
Also, it follows that the space $\mathbf{SC}$ is free. The cofreeness is also guaranteed, following arguments analog to the ones exposed in \cite{bergeron_98}, but since this is not too relevant for our main results we skip the details.

We have that $\mathbf{SC}^D$ is graded, connected, has a unit $\kappa_{\emptyset}\in\mathbf{SC}^D_0$ and a counit $\epsilon:\mathbf{SC}\rightarrow\mathbb C$ obtained by taking the coefficient of $\kappa_{\emptyset}.$ In order to get a bialgebra structure, as stated in the preliminaries, most of the compatibilities coming from the requirement on the maps $m,u,\Delta,\epsilon$ are straightforward to check. The compatibility between the product and the coproduct is less obvious and is what will allow us to conclude the main result of this paper. Namely, we want to prove that  $\Delta (P_{\mu}\cdot P_{\nu}) = \Delta(P_{\mu})\cdot\Delta(P_{\nu})$. Now we are ready to prove the main theorem.
\begin{theorem}
The product and coproduct given in proposition \ref{pbasis}, provides the space $\textbf{SC}^D$ with a Hopf algebra structure.
\end{theorem}

\begin{proof}
We prove only the compatibility relation between the product and the coproduct as explained in the previous paragraph. Let $\lambda\in D_{2k}(q),\mu\in D_{2(n-k)}(q)$, then
\begin{align*}
\Delta(P_{\lambda})\cdot\Delta(P_{\mu}) &= \left(\displaystyle\sum_{\substack{\lambda = \tau_1\sqcup\sigma_1\\ B\subseteq[k]}}P_{\text{st}_B(\tau_1)}\otimes P_{\text{st}_{B^c}(\sigma_1)}\right)\left(\displaystyle\sum_{\substack{\mu =\tau_2\sqcup\sigma_2\\ C\subseteq[n-k]}}P_{\text{st}_C(\tau_2)}\otimes P_{\text{st}_{C^c}(\sigma_2)}\right)\\
&= \displaystyle\sum_{\substack{\lambda=\tau_1\sqcup\sigma_1\\B\subseteq[k]}}\displaystyle\sum_{\substack{\mu=\tau_2\sqcup\sigma_2\\C\subseteq[n-k]}}P_{\text{st}_B(\tau _1)}P_{\text{st}_C(\tau _2)}\otimes P_{\text{st}_{B^c}(\sigma _1)}P_{\text{st}_{C^c}(\sigma_2)}\;\;\;\;\;\;\;\;\;\;\;\;\;\;\;\;\;\;\;\;\;\;\;\;\;\;(3.9)\\
&= \displaystyle\sum_{\substack{\lambda=\tau_1\sqcup\sigma_1\\B\subseteq[k]}}\displaystyle\sum_{\substack{\mu=\tau_2\sqcup\sigma_2\\C\subseteq[n-k]}}P_{\text{st}_B(\tau _1)\sqcup\text{ st}_C(\tau_2\uparrow^{|B|})}\otimes P_{\text{st}_{B^c}(\sigma _1)\sqcup\text{st}_{C^c}(\sigma_2\uparrow^{|C|})}
%&= \displaystyle\sum_{\substack{\lambda\sqcup\mu=\tau\sqcup\sigma\\ A\subseteq [\pm n],\; A=-A}}P_{\text{st}_A(\tau)}\otimes P_{\text{st}_{A^c}(\sigma)} \text{   where   }\tau=\tau_1\sqcup\tau_2,\;\sigma=\sigma_1\sqcup\sigma_2\\
%&= \Delta(P_{\lambda\sqcup\mu}) = \Delta(P_{\lambda}\cdot P_{\mu})
\end{align*}
On the other hand, we have
\begin{align*}
\Delta(P_{\lambda}\cdot P_{\mu}) &=  \Delta(P_{\gamma})\;\;\;\;\;\;\;\;\;\;\;\;\;\;\;\;\;\;\;\;\;\;\;\;\;\;\;\;\;\;\;\;\;\;\;\;\;\text{ where }\gamma=\lambda\sqcup\;\mu\uparrow^{k}\;\;\;\;\;\;\;\;\;\;\;\;\;\;\;\;\;\;\;\;\;\;\;(3.10)\\
&= \displaystyle\sum_{\substack{\gamma = \tau\sqcup\sigma\\ A\subseteq[n]}}P_{\text{st}_A(\tau)}\otimes P_{\text{st}_{A^c}(\sigma)}\;\;\;\;\;\;\;\;\;\;\;\;\;\text{ with }\gamma|_{A}=\tau,\;\gamma|_{A^c}=\sigma\;\;\\
\end{align*}
\noindent Now, since $\gamma=\lambda\sqcup\mu\uparrow^{k}$, then we can decompose $\tau$ and $\sigma$ in the form
$$
\tau=\tau_1\sqcup\tau_2\hspace{3cm}
\sigma=\sigma_1\sqcup\sigma_2
$$
such that $\tau_1,\sigma_1$ only intersect $\lambda$ and similarly, $\tau_2,\sigma_2$ only intersect $\mu\uparrow^{k}$. This decomposition induces a decomposition on the set $A=B\sqcup C$, so that the last equality in $(3.10)$ becomes
\begin{align*}
\Delta(P_{\lambda}\cdot P_{\mu}) &= \displaystyle\sum_{\substack{\gamma =( \tau_1\sqcup\tau_2)\sqcup(\sigma_1\sqcup\sigma_2)\\ B\sqcup C\subseteq[n],}}P_{\text{st}_{B\sqcup C}(\tau_1\sqcup\tau_2)}\otimes P_{\text{st}_{B^c\sqcup C^c}(\sigma_1\sqcup\sigma_2)}\\
&= \displaystyle\sum_{\substack{\gamma =( \tau_1\sqcup\tau_2)\sqcup(\sigma_1\sqcup\sigma_2)\\ B\sqcup C\subseteq[n],}}P_{\text{st}_B(\tau_1)\sqcup\text{st}_C(\tau_2\uparrow^{|B|})}\otimes P_{\text{st}_{B^c}(\sigma_1)\sqcup\;\text{st}_{C^c}(\sigma_2\uparrow^{|B^c|})}\\
&= \displaystyle\sum_{\substack{\lambda=\tau_1\sqcup\sigma_1\\B\subseteq[k]}}\displaystyle\sum_{\substack{\mu=\tau_2\sqcup\sigma_2\\C\subseteq[n-k]}}P_{\text{st}_B(\tau _1)\sqcup\text{ st}_C(\tau_2\uparrow^{|B|})}\otimes P_{\text{st}_{B^c}(\sigma _1)\sqcup\;\text{st}_{C^c}(\sigma_2\uparrow^{|C|})}
\end{align*}
Putting together this last equality with equation $(3.9)$, we can conclude that the desired compatibility holds.
\end{proof}

This allows us to conclude that the space \textbf{SC}$^D$ is indeed a combinatorial Hopf algebra as defined in the preliminaries.

We want to point out that different definitions of a combinatorial Hopf algebra can be given depending on the purposes. An alternative definition is as follows. A Hopf algebra $\mathcal A$ is a combinatorial Hopf algebra if it is graded, connected and has a distinguished character $\zeta:\mathcal A\rightarrow \mathbb K$. This singled out character is given by the trivial character in the case when the Hopf structure on $\mathcal A$ arises from representation theory (see \cite{abs_hopf}). 

Notice that we have not given an explicit computation of the product in  \textbf{SC}$^D$ in terms of the supercharacter basis. That is, what are the structure coefficients of the supercharacter basis? Can we find a meaning for their structure constants from the point of view of representation theory? This is still a question that awaits for an answer.  The author believed that the structure constants in the supercharacter basis of type $D$ would resembled the behaviour of type $A$ but S. Andrews has proved this is not the case.

Looking at a small aspect of this question would be to compute the degree of the product of supercharacters. More precisely we have that for $\lambda\in D_{2n}(q)$ and $\mu\in D_{2m}(q)$
\begin{equation*}
\chi^{\lambda}(1)\cdot\chi^{\mu}(1)=q^{2m\alpha}\chi^{\lambda\sqcup\mu\uparrow n}(1)
\end{equation*}
\noindent where $\alpha=|\{\small{i}\larc{}\small{j}\in\lambda^+\;:\; n\prec j\}|$. 
In other words, the degree of the product is the product of the degrees up to a factor of $q$. 

We finish this paper by giving a brief outline concerning the types $B$ and $C$.
\section{Final comments}

\subsection{Type $B$ and type $C$}

Following the construction in \cite{andre_2}, supercharacters and superclasses for types $B$ and $C$ are also indexed by labelled partitions of the corresponding type. The unipotent upper triangular matrices of type $B$ is the group of $m\times m$ orthogonal matrices where $m=2n+1$ for some $n\in\mathbb Z_{\geq 0}$. We define $B_{m}(q)$-\emph{partitions} as labelled set partitions on the set $\{1,\dots , n,0,-n,\dots , -1\}$ satisfying the same properties as $D_{2n}(q)$-partitions with the additional property that we allow at most one arc of the form $\small{i}\larc{a}\small{0}$ together with $\small{0}\larc{-a}\small{-i}$. 

\noindent Unfortunately, any attempt from the author to construct a product on $\textbf{SC}^B$ fails since dealing with odd-size matrices make impossible an embedding from $\textbf{SC}^B_{2k+1}\times\textbf{SC}^B_{2l+1}$ to $\textbf{SC}^B_{2(k+l)+1}$, although a different grading and suitable changes could make it possible. On the other hand, we have a structure of $\textbf{SC}^D$-module on $\textbf{SC}^B$, since it is clear that $\textbf{SC}^D_{2k}\times\textbf{SC}^B_{2l+1}$ embeds into $\textbf{SC}^B_{2(k+l)+1}\in\textbf{SC}^B$.

For unipotent upper triangular matrices of type $C$ the situation is better. This type corresponds to the group of $2n\times 2n$ symplectic matrices and the combinatorial description for its supercharacter theory resembles the one for type $D$. In this case $C_{2n}(q)$-\emph{partitions} are defined as in Definition \ref{D-partition} but we also allow arcs ${\small{i}}\larc{a}\small{-i}$.  Similar arguments can be used in this case, producing a similar definition for product and coproduct over the graded vector space $\textbf{SC}^C$ endowing it with a Hopf algebra structure.

\subsection{Recent and forthcoming work}\label{future_work}

We remind that this paper has considered only the case when $\text{char}(\mathbb F_q)\geq 3$. The case $p=2$ requires a different description of the elements of the group $U^D_{2n}(q)$. We want to understand this case as well, since this might allow us to have an unlabelled version of what we have done here. 

On the other hand, a coarser version of the supercharacter theory of type $D$ as exposed here could have some connection with the case $q=2$. Namely, by lumping together conjugacy classes in $U^D_{2n}(q)$ through the action $(T_{2n}U_{2n}(q)AU_{2n}(q)T_{2n}^{-1}+I_{2n})\cap U^D_{2n}(q)$, where $T_{2n}$ is the subgroup of diagonal matrices of $GL_{2n}(q), A\in \mathfrak u^D(q)$, gives a coarser superclass theory, whose supercharacter values are integers. Hence, the unlabelled version of the Hopf algebra constructed here would realize the version given by this super-theory. This is inspired by the work done in \cite{bt}.

Finally, we want to point out that types $C$ and $D$ not only have a Hopf algebra structure, a Hopf monoid structure can be provided too. For a basic background on Hopf monoid in species the reader can consult \cite{aguiar:2011}. 

\noindent Briefly, let the species $\textbf{SC}^{D}$ be such that for a finite set $K$
$$
\textbf{SC}^{D}[K]=\displaystyle\oplus_{\phi\in L[K]}\textbf{SC}^{(\phi,D)}[K]
$$
where $L[K]$ is the set of linear orders on $K$ and $\textbf{SC}^{(\phi,D)}[K]$ being the set of $D_{2|K|}(q)$-partitions that respect the order given by $\phi$. In other words, let the set $K\sqcup\bar K$ be ordered by $\phi\cdot\overleftarrow\phi$ where $\cdot$ denotes concatenation and $\bar K$ is a second copy of $K$ with $\overleftarrow\phi$ being the order of $K$ reversed. Now, after drawing the arcs of $\lambda$ on top of $K\sqcup\bar K$ and putting $\bar\phi = \phi\cdot\overleftarrow\phi$ we ask that if $\small{i}\larc{a}\small{j}\in\lambda$ then $i\leq_{\bar\phi} j$; also $\lambda$ must satisfy the analog of condition $(b)$ in Definition \ref{D-partition}, replacing $\prec$ by $\leq_{\bar\phi}$.
Then we can check that for nonintersecting finite sets $I,J$ the following maps
\begin{align*}
m_{I,J} &:  \textbf{SC}^{(\phi,D)}[I]  \otimes  \textbf{SC}^{(\tau,D)}[J]\rightarrow\textbf{SC}^{(\phi\cdot\tau,D)}[I\sqcup J]\\
\Delta_{I,J} &: \textbf{SC}^{(\phi,D)}[I\sqcup J]\rightarrow  \textbf{SC}^{(\phi|_{I},D)}[I]\otimes \textbf{SC}^{(\phi|_{J},D)}[J]
\end{align*}
defined in analogy with the structure presented here, satisfy all the axioms required to make of the species $ \textbf{SC}^{D}$ a Hopf monoid. All of this is done following \cite{aguiar:2011} and it is part of the author's Ph.D. thesis. 

\section{Acknowledgements}
I would like to express my gratitude to my supervisor Nantel Bergeron for introducing me into the world of supercharacters and Hopf algebras. Also, I thank to Carlos Andr\'e, Nat Thiem and Scott Andrews for fruitful conversations and enlightening comments. Finally, many thanks to the reviewers for their significant comments that shed light on my work.

\footnotesize{

}

\end{document}